\documentclass[reqno]{amsart}

\usepackage{amsmath,amssymb,amscd, accents} \usepackage{graphicx}
\DeclareGraphicsExtensions{.eps} \usepackage{mathrsfs}
\usepackage[mathcal]{eucal}

\newtheorem{Thm}{Theorem}{\bfseries}{\itshape}
\newtheorem*{Thm*}{Theorem}{\bfseries}{\itshape}
\newtheorem{Cor}{Corollary}{\bfseries}{\itshape}
\newtheorem{Prop}[Cor]{Proposition}{\bfseries}{\itshape}
\newtheorem{Lem}[Cor]{Lemma}{\bfseries}{\itshape}
\newtheorem*{Lem*}{Lemma}{\bfseries}{\itshape}
{\bfseries}{\itshape}
\newtheorem{Conj}[Cor]{Conjecture}{\bfseries}{\itshape}
\newtheorem{Def}[Cor]{Definition}{\bfseries}{\rmfamily}
{\scshape}{\rmfamily}
\newtheorem{Rem}[Cor]{Remark}{\scshape}{\rmfamily}

\renewcommand\ge{\geqslant} \renewcommand\le{\leqslant}
\let\tildeaccent=\~ \let\hataccent=\^
\renewcommand\~[1]{\widetilde{#1}}
\def\o{\accentset{\circ}}

\def\<{\left<} \def\>{\right>} \def\({\left(} \def\){\right)}

\def\abs#1{\left\vert #1 \right\vert}  

\let\parasymbol=\S \def\secref#1{\parasymbol\ref{#1}}
 \def\pd#1#2{\tfrac{\partial#1}{\partial#2}}

\let\polishL=l \def\Zoladek.{\.Zol\c adek}

\def\codim{\operatorname{codim}}

 \def\ord{\operatorname{ord}}
 \def\etc.{\emph{etc}.}

\def\:{\colon} \def\R{{\mathbb R}} \def\C{{\mathbb C}} \def\Z{{\mathbb
    Z}} \def\N{{\mathbb N}}  
 \def\e{\varepsilon} \def\S{\varSigma}

 \def\d{{\mathrm d}}

 \let\PolishL=\L \def\Lojas.{\PolishL ojasiewicz}

\def\cF{{\mathcal F}}  
  
\def\cC{{\mathcal C}} \def\cM{{\mathcal M}}

 \def\mult{\operatorname{mult}}

\def\rest#1{{\vert_{#1}}}

\def\vell{{\boldsymbol\ell}}

\def\vH{{\mathbf H}}
\def\vL{{\mathbf L}}
\def\vc{{\mathbf c}}
\def\w{\omega}

\def\degf{\mathfrak{D}}
\def\fg{\mathfrak{g}}

\def\wt{\operatorname{wt}}

\def\_#1{_{1\ldots #1}}

\begin{document}

\title{Pfaffian Intersections and Multiplicity Cycles}

\author{Gal Binyamini}
\address{University of Toronto}
\email{galbin@gmail.com}
\thanks{The author was supported by the Banting
Postdoctoral Fellowship and the Rothschild Fellowship}

\begin{abstract}
  We consider the problem of estimating the intersection multiplicity
  between an algebraic variety and a Pfaffian foliation, at every
  point of the variety. We show that this multiplicity can be
  majorized at every point $p$ by the local algebraic multiplicity at
  $p$ of a suitably constructed algebraic cycle. The construction is
  based on Gabrielov's complex analog of the Rolle-Khovanskii lemma.

  We illustrate the main result by deriving similar uniform estimates
  for the complexity of the Milnor fiber of a deformation (under a
  smoothness assumption) and for the order of contact between an
  algebraic hypersurface and an arbitrary non-singular one-dimensional
  foliation. We also use the main result to give an alternative
  geometric proof for a classical multiplicity estimate in the context
  of commutative group varieties.
\end{abstract}

\date{\today}

\maketitle

\section{Introduction}
\label{sec:intro}

\subsection{Motivation}

Let $M$ be a complex variety of dimension $n$. For simplicity of the
presentation we assume that $M=\C^n$. We consider problems of the
following type:
\begin{enumerate}
\item[I.] Given a polynomial vector field $\xi$ on $M$ and an algebraic
  hypersurface $V\subset M$, estimate for each $p\in V$ the order of
  contact between $V$ and the trajectory of $\xi$ through $p$.
\item[II.] More generally, given a foliation $\cF$ on $M$ of codimension $k$
  and an algebraic variety $V\subset M$ of dimension $k$, estimate
  for each $p\in V$ the intersection multiplicity between $V$ and the
  leaf of $\cF$ through $p$.
\item[III.] Given a flat algebraic family $V\subset M\times\C$, estimate for
  each point $p\in M$ the topological complexity (for instance, sum of
  Betti numbers) of the Milnor fiber of $V$ at $(p,0)$.
\end{enumerate}
These questions all share a similar nature. Namely, there is some map
$\mu:M\to\Z_{\ge0}$ defined in algebraic terms, which one would like
to estimate from above. Moreover, each of the questions involves
algebraic degrees (e.g. the degrees of the polynomials defining
$\xi,\cF,V$) and one often hopes to obtain an estimate in terms of the
degrees of this data. We assume for simplicity that all degrees
are bounded by a number $d$.

Various estimates for each of these problems have been studied (see
the appropriate subsection of~\secref{sec:applications} for
references). The most common form for an estimate of this type is
\begin{equation} \label{eq:estimate-no-p}
  \mu(p) \le \phi(d)
\end{equation}
where $\phi(d)$ is some function depending on $d$ and the discrete
parameters of the problem (such as the dimensions $n$ and $k$ in
problems I--III above). Motivated by the Bezout theorem, it is
reasonable to hope that $\phi$ can be taken to be a polynomial of
degree $n$ in $d$, and indeed this is often the case. Moreover, this
can easily be shown to be optimal (up to a multiplicative constant)
for each of the problems above if one considers only bounds of the
form~\eqref{eq:estimate-no-p}.

It is sometimes desirable to have a more detailed understanding of
the behavior of $\mu(p)$ as $p$ varies over a set of points. One may
hope that in this context more refined estimates can be given. For
instance, it is often reasonable to expect that outside a set of
codimension $k$ an estimate of order $d^{k-1}$ holds. A result of this
nature was proved for problem I in \cite{nesterenko:mult-nonlinear},
motivated by problems in transcendental number theory (following
similar but less refined estimates in \cite{bm:mult-I,bm:mult-II}).

We propose an algebraic mechanism which is useful in the description
of such phenomena. Recall that an algebraic cycle $\Gamma$ is a linear
combination of algebraic varieties with integer (for us also positive)
coefficients. The multiplicity $\mult_p V$ for an irreducible variety
$V$ is defined to be the multiplicity of the intersection at $p$
between $V$ and a generic linear space of complementary dimension.
This is extended to the multiplicity $\mult_p\Gamma$ for any cycle. An
estimate in terms of a \emph{multiplicity cycle} is an estimate of the
form
\begin{equation} \label{eq:estimate-gamma}
  \mu(p) \le \mult_p\Gamma
\end{equation}
for some cycle $\Gamma$. Again motivated by the Bezout theorem, one
may hope that the $k$-codimensional part of $\Gamma$, denoted
$\Gamma^k$, satisfies $\deg\Gamma^k=O(d^k)$. We prove the existence of
an estimate of this form for problem II above, under the assumption
that the foliation is Pfaffian (see~\secref{sec:mult-foliations} for
discussion and a conjecture regarding the general case). We then use
this result to derive similar estimates for problems I in full
generality and III under a smoothness assumption.

An estimate of the form~\eqref{eq:estimate-gamma} has convenient
algebraic properties which enable a broader range of applications than
the uniform estimate~\eqref{eq:estimate-no-p}. We illustrate this by
giving an alternative proof for a multiplicity estimate on group
varieties due to Masser and W\"ustholz
in~\secref{sec:group-estimates}.

Finally, we make a general remark on the form of the
estimate~\eqref{eq:estimate-gamma}. We were first able to obtain an
estimate of this form using the topological methods described in this
paper. Later, using this expression as a part of an inductive
hypothesis, we were also able to obtain a similar result for problem I
using an entirely different algebraic approach (see
Remark~\ref{rem:mult-sing} for a brief discussion on the relative
advantages of each approach). We view this as an indication that
estimates in terms of multiplicities of cycles are fundamentally
suitable for the treatment of problems of this type.

\subsection{Statement of the main result}

Let $M$ denote a complex variety of dimension $n$. Let
$\w_1,\ldots,\w_k\in\Lambda^1(M)$ be an ordered sequence of
holomorphic one-forms. We will say that $\w_1,\ldots,\w_k$ defines an
\emph{integrable Pfaffian system} at a point $p\in M$ if
$\w_1\wedge\cdots\wedge\w_k$ does not vanish at $p$ and there exists a
(necessarily unique) chain of germs of manifolds
$S^k_p\subset\cdots\subset S^1_p\subset M$ containing $p$, where each
inclusion is of codimension $1$ and such that $\w_i\rest{S^i}\equiv0$
for $i=1,\ldots,k$.

Let $\C_e$ denote the germ of $\C$ at the origin, and $e$ the
coordinate function on $\C_e$. A function $f:M\times \C_e\to\C$ will
be called a \emph{family of analytic functions} on $M$ depending on
the parameter $e$. We adopt the notation $f(x,e)\equiv f^e(x)$. Recall
that a variety $V\subset M\times\C_e$ is said to be \emph{flat} if
each of its components project dominantly on $\C_e$. We denote by
$\cF(V)$ the flat family obtained from $V$ by removing any component
violating this condition. We denote by $V^\e:=V\cap\{e=\e\}$ the
$\e$-fiber of $V$.

To simplify the notation we denote $\w\_k:=\w_1,\ldots,\w_k$ and
similarly for a tuple of families of analytic functions $f\_m$. For
brevity, we also denote $\bigwedge\w\_k:=\w_1\wedge\cdots\wedge\w_k$.

\begin{Def}
  For a family $V\subset M\times\C_e$ and $p\in M$, we define the
  \emph{deformation multiplicity} of $V$ at $p$, denoted
  $\mult_p^e(V)$, to be the number of isolated points in
  $V^\e,\e\neq0$ converging to $p$ as $\e\to0$.

  Let $\w\_k$ be a Pfaffian system integrable at $p\in M$ and $f\_m$
  be families of meromorphic functions defined near $p$, with
  $p+m=\dim V-1$. We define the deformation multiplicity
  \emph{relative to $f\_m$ and $\w\_k$} to be
  \begin{equation}
    \mult_p^e(f\_m;\w\_k;V) := \mult_p^e(\{f\_m=0\}\cap S^k_p\cap V).
  \end{equation}
  In this case we count each isolated point of intersection (for fixed
  $\e$) with its associated multiplicity in the intersection of the
  cycles $[S^k_p]$ and $[f\_m=0]$ on $V^\e$ (which may be negative if
  $f\_m$ have poles).

  We omit $e$ when it is clear from the context, $V$ if it is
  $M\times\C_e$ and $\w\_k$,$f\_m$ if they are empty sequences.
\end{Def}

Recall that a (mixed) algebraic cycle in $M$ is a linear combination
with integer coefficients of subvarieties of $M$. We will only
consider cycles with non-negative coefficients. For a mixed cycle
$\Gamma$, we denote by $\Gamma^j$ the $j$-codimensional part of
$\Gamma$. The (algebraic) multiplicity of a cycle $\Gamma$ of pure
dimension $j$ at a point $p$, denoted $\mult_p\Gamma$, is defined to
be the multiplicity of the intersection between $\Gamma$ and a generic
linear subspace of codimension $j$ passing through $p$. This is
extended to arbitrary mixed cycles by linearity.

Our goal is to majorize the multiplicity $\mult_p(f\_{n-k};\w\_k)$ as
a function of $p$ in terms of algebraic multiplicities. For simplicity
of the presentation we assume that the ambient space is given by
$M=\C^n$ and $f\_{n-k},\w\_k$ are polynomial (although the
construction could clearly be carried out for more general ambient
spaces). Our main result is as follows.

\begin{Thm}\label{thm:mult-cycle}
  There exist a mixed algebraic cycle $\Gamma:=\Gamma(f\_{n-k};\w\_k)$
  of top dimension $k$ in $M$ such that for any point $p$ where
  $\w\_k$ is integrable,
  \begin{equation}
    \mult_p(f\_{n-k};\w\_k) \le \mult_p\Gamma(f\_{n-k};\w\_k).
  \end{equation}
  Moreover, $\Gamma$ is obtained algebraically from $f\_{n-k};\w\_k$
  as described in~Definition~\ref{def:mult-cycle}. We refer to
  $\Gamma$ as the \emph{multiplicity cycle} associated to
  $f\_{n-k};\w\_k$.
\end{Thm}

We denote by $\deg(f\_{n-k};\w\_k):=(\beta\_{n-k};\alpha\_k)$
where $\deg f_i=\beta_i$ and $\deg\w_j=\alpha_j$ (where the
degree of a one-form is understood to be the maximal degree
of any of its coefficients in the standard coordinates).

\begin{Thm}\label{thm:mult-cycle-degs}
  Let $\deg(f\_{n-k};\w\_k):=(\beta\_{n-k};\alpha\_k)$. Then
  \begin{equation}
    \deg\Gamma^{n-j}(f\_{n-k},\w\_k) \le \frac{k!}{j!} 2^{(k-j)(k-j-1)/2} \beta_1\cdots\beta_{n-k} S^{k-j}
  \end{equation}
  where
  \begin{equation}
    S := \alpha_1+\cdots+\alpha_k+\beta_1+\cdots+\beta_{n-k}.
  \end{equation}
\end{Thm}

We sometimes consider the degrees of $f\_{n-k},\w\_k$ as the main
asymptotic while fixing all other parameters. We give another form of
the bound which stresses this particular asymptotic.

\begin{Cor}\label{cor:mc-degs-simple}
  Let the degrees of $f\_{n-k}$ and $\w\_k$ be bounded by an integer
  $d$. Then
  \begin{equation}
    \deg\Gamma^{n-j}(f\_{n-k},\w\_k) \le C_{n,k,j} d^{n-j}
  \end{equation}
  where
  \begin{equation}
    C_{n,k,j}:=\frac{k!}{j!} 2^{(k-j)(k-j-1)/2} n^{k-j}.
  \end{equation}
\end{Cor}

\subsection{Structure of this paper}

In~\secref{sec:background} we survey the Rolle-Khovanskii lemma for
real Pfaffian systems and its basic application in the theory of real
Fewnomials; and the complex analog of this lemma due to Gabrielov.
In~\secref{sec:mc} we present the construction of the multiplicity
cycles for Pfaffian systems and prove the main theorems of the paper.
In~\secref{sec:applications} we discuss various application of the
main theorem --- to the topological complexity of Milnor fibers
in~\secref{sec:mf-estimates}; to multiplicity estimates for
trajectories of non-singular vector fields
in~\secref{sec:vf-estimates}; and to multiplicity estimates on group
varieties in~\secref{sec:group-estimates}. Finally,
in~\secref{sec:sc-bounds} we present an auxiliary compactness result
for multiplicity cycles which is used throughout the paper.

\section{Background}
\label{sec:background}

\subsection{Real Pfaffian systems}

Integrable systems of Pfaffian equations were first studied by
Khovanskii \cite{Khovanskii:Fewnomials} in the real domain, giving
rise to the theory of Fewnomials. We briefly recall the basic elements
of this theory in a context suitable for comparison with the present
paper. Our presentation is thus restricted in scope, and we refer the
reader to \cite{Khovanskii:Fewnomials} for a complete account.

Let $M$ be a real manifold of dimension $n$ and $\w\in\Lambda^1(M)$ a
one-form. We say that $S$ is a \emph{separating solution} for $\w$
with the coorientation defined by $\w$, if:
\begin{enumerate}
\item $S$ has codimension $1$ and $\w\rest S\equiv0$.
\item There exists a manifold $F\subset M$, called the \emph{film},
  such that $\partial F=S$ and $\w$ takes positive values on the
  vector pointing inside $F$ at every point of $S$.
\end{enumerate}

Let $f\_{n-1}$ be smooth functions, and $\Gamma$ their set of common
zeros. We assume for simplicity that $\Gamma$ is a smooth complete
intersection curve meeting $S$ transversally\footnote{this restriction
  can be relaxed significantly by a perturbation argument which we
  omit for simplicity}. The following result forms the basis of the
theory of real Fewnomials.

\begin{Lem}[Rolle-Khovanskii, \protect{\cite[III.4 Corollary~2]{Khovanskii:Fewnomials}}] \label{lem:rolle-khovnaskii}
  Let $B$ denote the number of non-compact components of $\Gamma$
  and $N$ denote the number of zeros of $\w\rest\Gamma$. Then the number
  of points in $S\cap\Gamma$ is bounded by $B+N$.
\end{Lem}
\begin{proof}
  Choose any orientation for $\Gamma$, and let $p,q\in\Gamma\cap S$ be
  two adjacent points of intersection (i.e. such that the piece of
  $\Gamma$ lying between $p$ and $q$, denoted $\Gamma_{p,q}$, does not
  contain points of $S$). It follows that the positive tangent vector
  to $\Gamma$ points inside the film $F$ at $p$ and outside the film
  $F$ at $q$ or vice versa. By condition (2) above, this means that
  $\w$ changes sign between $p$ and $q$, hence by the classical Rolle
  lemma must have a zero on $\Gamma$ between the two. A simple
  counting argument concludes the proof.
\end{proof}

Assume now that $M=\R^n$. In this case the number of non-compact
components $B$ in Lemma~\ref{lem:rolle-khovnaskii} can also be
algebraically estimated as follows.

\begin{Lem}[\protect{\cite[Lemma on page 11]{Khovanskii:Fewnomials}}]\label{lem:real-components}
  The exists an affine hyperplane $H$ such that the number of
  non-compact components of $\Gamma$ does not exceed the number
  of intersections between $\Gamma$ and $H$.
\end{Lem}

Consider now a sequence of one-forms $\w\_k\in\Lambda^1(M)$ and a
chain of submanifolds $S^k\subset\cdots\subset S^1\subset S^0=M$ where
each inclusion $S^{i+1}\subset S^i$ forms a separating solution for
$\w_i$. Let $f\_{n-k}$ be a sequence of polynomials, which we shall
assume to be sufficiently generic\footnotemark[1].
Lemmas~\ref{lem:rolle-khovnaskii} and~\ref{lem:real-components}
suggest the following estimate for the number of intersections between
$S^k$ and $f\_{n-k}=0$,
\begin{multline} \label{real-few-induction}
  \#(S^k\cap f\_{n-k}) \le \#(S^{k-1}\cap\{f\_{n-k}=g_{n-k+1}=0\})\\
  +\#(S^{k-1}\cap\{f\_{n-k}=H=0\})
\end{multline}
where
\begin{equation}
  g := \frac{\bigwedge\w\_k \wedge \bigwedge\d f\_{1-k}}{\bigwedge\d x\_n}.
\end{equation}
and $H$ is the hyperplane whose existence is guaranteed by
Lemma~\ref{lem:real-components}. It is sometimes customary to include
$H$ in the definition of $g$ and avoid the second summand
in~\eqref{real-few-induction}. However, for the purposes of comparison
with the present paper the form above is more convenient. In
particular, a simple inductive argument gives the following.

\begin{Thm}\label{thm:real-mult-cycle}
  There exist a mixed algebraic cycle $\Gamma:=\Gamma(f\_{n-k},\w\_k)$
  in $M$ such that
  \begin{equation}
    \#(S^k\cap f\_{n-k}) \le \deg\Gamma(f\_{n-k},\w\_k).
  \end{equation}
  Moreover, $\Gamma$ is obtained algebraically from $f\_{n-k},\w\_k$.
\end{Thm}
\begin{proof}[Sketch of proof]
  After $k$ inductive applications of~\eqref{real-few-induction}, one
  obtains various summands of the form
  \begin{equation}
    \#(\{h_1=\cdots=h_{n-r}=H_1=\cdots=H_r=0\})
  \end{equation}
  where $h_i$ are polynomials and $H_i$ are linear functions.
  This quantity is certainly bounded by $\deg\Gamma_h$ where
  $\Gamma_h:=\{h_1=\cdots=h_{n-r}=0\}$, and collecting all such cycles
  into $\Gamma$ we obtain the estimate.
\end{proof}

It may appear unclear why one should consider components of different
dimensions in $\Gamma(f\_{n-k},\w\_k)$ when it is equally possible to
replace each of the positive dimensional components $\Gamma_h$ in the
proof above by its zero-dimensional intersection with
$H_1=\cdots=H_r=0$. The reader may note the formal analogy between
Theorems~\ref{thm:mult-cycle} and~\ref{thm:real-mult-cycle}. We will
see that in the local case, the decomposition into components of
various dimensions plays a more principal role.

\subsection{Complex Pfaffian systems}

In the complex setting one can no longer expect a global estimate of
the type given in Theorem~\ref{thm:real-mult-cycle}. However, a local
analog has been developed by Gabrielov \cite{gabrielov:MultPfaffian}.
This was later used to establish various estimates on the geometric
complexity of sets defined using Pfaffian functions (see
\cite{GV:PfaffianComplexity} for a survey). We recall the fundamental
result from this paper, namely a local complex analog of
Lemma~\ref{lem:rolle-khovnaskii}.

Let $M$ be a complex manifold of dimension $n$ and $\w\in\Lambda^1(M)$
a one-form. Let $p\in M$ and suppose that $\w$ admits an integral
manifold $S_p$ through $p$. Finally, let $X\subset M\times\C_e$ be a
(reduced) flat family, $\dim X=2$. We fix any analytic coordinate
system $x_1,\ldots,x_n$ around $p$.

\begin{Lem}[\protect{\cite[Theorem 1.2]{gabrielov:MultPfaffian}}] \label{lem:rolle-gabrielov}
  Assume that in a neighborhood of $p$, $X^\e\cap S_p$ consists of
  isolated points for small $\e\neq0$. Let $\d H$ be a generic
  constant one-form. Then
  \begin{equation}\label{eq:gab-rolle}
    \mult_p(\w;X) \le \mult_p(g;X)+\mult_p(\d H;X)
  \end{equation}
  where
  \begin{equation} \label{eq:gab-g}
    g = \frac{ \w\wedge (\d e+ce\d H)\rest X }
             {\d H\wedge\d e \rest X}
  \end{equation}
  and $c$ is a generic complex number.
\end{Lem}

We remark that the second summand in~\eqref{eq:gab-rolle} could be
replaced by the equivalent term $\mult_p(H;X)$, and this is the
formulation originally appearing in \cite{gabrielov:MultPfaffian}.
With this formulation, one is required to choose a generic $H$
\emph{vanishing at $p$}, whereas in the formulation above we obtain an
expression essentially uniform over $p$. However, the one-form $\d H$
must still satisfy a genericity condition possibly depending on the
point $p$, and we shall have to resolve this technical difficulty in
order to obtain truly uniform estimates over $p$.

The extra generic factor $ce\d H$ in Lemma~\ref{lem:rolle-gabrielov}
is needed in order to avoid certain degeneracies in the case where the
intersection $S_p\cap X^0$ degenerates into a non-isolated
intersection. We sketch the proof for the case where the intersection
is isolated and the generic fiber $X^\e$ is smooth. We also assume for
simplicity that $\w$ is a closed form, and hence hence $\w=\d\pi$ for
some analytic function $\pi:M\to\C$ around $p$. In this case case one
may take $c=0$, and we will in fact show equality
in~\eqref{eq:gab-rolle}. Our presentation follows that of Gabrielov
\cite{gabrielov:MultPfaffian}.

\begin{proof}[Sketch of proof]
  By assumption the fibers $X^\e$ are analytic curves. We fix a small
  positive $\delta$ and denote $D_\delta=\{\abs{z}\le\delta\}$.
  Consider the fiber
  \begin{equation}
    F^\e:=X^\e\cap\pi^{-1}(D_\delta).
  \end{equation}
  In fact, for $0<\e\ll\delta$ this fiber has the homotopy type of the
  Milnor fiber of $X$ at $p$.

  Let $\mu:=\mult(\w;X)$. Since the intersection is isolated at $e=0$
  by assumption, this means that $\pi:F^0\to D_\delta$ is a ramified
  $\mu$ to $1$ map. For fixed $\delta$ and sufficiently small $\e$,
  this remains true for $\pi:F^\e\to D_\delta$ as well. By the
  Riemann-Hurwitz formula,
  \begin{equation} \label{eq:chi-pi}
    \chi(F^\e) = \mu - \mult_p(\w\rest{X^\e};X)
  \end{equation}
  where we slightly abuse notation and allow the top-form
  $\w\rest{X^\e}$ in place of a meromorphic function in the second
  summand, which corresponds to the number of critical points of $\pi$
  on $F^\e$ (with their multiplicities).

  Arguing similarly with $H$ in place of $\pi$, we have
  \begin{equation} \label{eq:chi-H}
    \chi(F^\e) = \mult_p(H;X) - \mult_p(\d H\rest{X^\e};X).
  \end{equation}
  Note that while $F^\e$ appearing in this equation is in fact
  different (being defined using $H$ in place of $\pi$), in homotopy
  type both sets agree with the Milnor fiber of $X$ at $p$. Thus we
  may compare the Euler characteristics from~\eqref{eq:chi-pi}
  and~\eqref{eq:chi-H} to obtain
  \begin{align*}
    \mu &= \mult_p(\w\rest{X^\e};X) - \mult_p(\d H\rest{X^\e};X)+\mult_p(H;X) \\
        &= \mult_p(g;X)+\mult_p(H;X)
  \end{align*}
  and the claim follows since $\mult_p(\d H;X)=\mult_p(H;X)$.
\end{proof}

Consider now a sequence of one-forms $\w\_k\in\Lambda^1(M)$ defining
an integrable Pfaffian system at $p\in M$, and the corresponding chain
of integral submanifolds
$S^k_p\subset\cdots\subset S^1_p\subset S^0_p=M$. Let $f\_{n-k}$ be
families of polynomials on $M$. Lemma~\ref{lem:rolle-gabrielov} can be
used in a manner analogous to the use of
Lemma~\ref{real-few-induction} in the proof of
Theorem~\ref{thm:real-mult-cycle} to obtain a bound for
$\mult_p(f\_{n-k};\w\_k)$ in terms of certain algebraic systems of
equations, and more specifically the number of solutions for these
equations converging to $p$ (see
\cite[Theorem~2.1]{gabrielov:MultPfaffian}).

\section{The multiplicity cycles}
\label{sec:mc}

Let $M=\C^n$. Let $\w\_k\in\Lambda^1(M)$ be a sequence of one-forms
with polynomial coefficients and $U\subset M$ a set such that $\w\_k$
defines an integrable Pfaffian system at $p$ for each $p\in U$, and
denote the corresponding chain of integral submanifolds by
$S^k_p\subset\cdots\subset S^1_p\subset S^0_p=M$. Let $f\_{n-k}$ be
families of polynomials on $M$.

In this section we describe the explicit construction of the
multiplicity cycle $\Gamma(f\_{n-k};\w\_k)$ and prove
Theorem~\ref{thm:mult-cycle}. We begin with a technical result on
smoothing deformations which was used in \cite[Proof of
Theorem~2.1]{gabrielov:MultPfaffian}.

\begin{Lem}\label{lem:smooth-deform}
  Let $p\in U$. For any $N\in N$ and $\vc=c\_{n-k}\in\C$ define the
  sequence $\~f\_{n-k}$ by $\~f_i=f_i-c_i e^N$. Then for sufficiently
  large $N$ and generic $\vc$ we have
  \begin{equation}\label{eq:f-vs-ft}
    \mult_p(f\_{n-k};\w\_k) \le \mult_p(\~f\_{n-k};\w\_k)
  \end{equation}
  and moreover, the intersection
  $\{\tilde f^\e\_{n-k}=0\}\cap S_p^{k-1}$ is an effectively
  non-singular curve intersecting $S_p^k$ discretely (in a
  neighborhood of $p$) for sufficiently small $\e\neq0$.
\end{Lem}
\begin{proof}
  Standard analytic arguments show that each isolated point of
  $f\_{n-k}=0$ on $S^k$ converging to $p$ as $e\to0$ survives the
  perturbation $f\to\~ f$ as long as $N$ is large enough (possibly
  bifurcating into a number of points of the same total multiplicity).
  This ensures~\eqref{eq:f-vs-ft}.

  To satisfy the non-singularity condition it is enough to verify that
  $c\_{n-k}\e^N$ is not a critical value for $f^\e\_{n-k}$ on
  $S^{k-1}_p$, which by the Bertini-Sard theorem is certainly true for
  generic $\vc$ and sufficiently small $\e\neq0$. Similarly one
  verifies that the intersection $\{\~f^\e\_{n-k}=0\}\cap S^k_p$ is
  discrete for generic $\vc$ and sufficiently small $\e\neq0$.
\end{proof}

The construction of the multiplicity cycle is based on an inductive
process, with the following consequence of
Lemma~\ref{lem:rolle-gabrielov} providing the key inductive step.

\begin{Lem} \label{lem:pfaff-rolle}
  Let $p\in U$ and let $\d H$ be a generic constant one-form. Then
  \begin{multline} \label{eq:pfaff-rolle}
    \mult_p(f\_{n-k};\w\_k) \le \mult_p(\~f\_{n-k},g;\w\_{k-1})\\
    +\mult_p(\~f\_{n-k};\d H,\w\_{k-1})
  \end{multline}
  with
  \begin{equation} \label{eq:pfaff-g}
    g = \frac{ \bigwedge \d\~f\_{n-k}\bigwedge\w\_k\wedge (\d e+ce\d H) }
             {\bigwedge \d x\_n\wedge\d e}
  \end{equation}
  where $\~f\_{n-k}$ are as given by Lemma~\ref{lem:smooth-deform} and
  $c$ is a generic complex number.
\end{Lem}
\begin{proof}
  We may apply Lemma~\ref{lem:smooth-deform} and assume without loss
  of generality that $f\_{n-k}$ are already in the prescribed form.
  Since $\d H_1$ is generic, its integral manifold through $p$
  intersects $S_p^1,\ldots,S_p^{k-1}$ transversally, and it follows
  that the sequence $\d H_1,\w\_{k-1}$ is an integrable Pfaffian
  system as well, hence~\eqref{eq:pfaff-rolle} is well defined.
  
  Define the flat family $X\subset M\times\C_e$ by
  \begin{equation}
    X := \cF\big[ \{f\_{n-k}=0\}\cap S^{k-1}_p \big].
  \end{equation}
  We can now apply Lemma~\ref{lem:rolle-gabrielov} to $X$ with the
  forms $\w_k,\d H$. By Lemma~\ref{lem:smooth-deform} the form
  \begin{equation}
    \bigwedge \d f\_{n-k} \bigwedge \w\_{k-1}
  \end{equation}
  is non-vanishing on $X$ (for small $\e\neq0$), and it follows that
  the zeros of the numerator of~\eqref{eq:pfaff-g} agree with those
  of~\eqref{eq:gab-g} for such $\e$. As the denominator
  of~\eqref{eq:pfaff-g} has no zeros,~\eqref{eq:pfaff-rolle} now
  follows from Lemma~\ref{lem:rolle-gabrielov}.
\end{proof}

Applying Lemma~\ref{lem:pfaff-rolle} iteratively $k$ times gives the
following.

\begin{Lem}\label{lem:mc-inductive-step}
  Let $p\in U$ and and let $\d H_1,\ldots,\d H_k$ be generic constant
  one-forms. Then
  \begin{equation}\label{eq:mc-inductive-step}
    \mult_p(f\_{n-k};\w\_k) \le \mult_p\Gamma+ \sum_{j=1}^{k-1} \mult_p(\~f\_{n-k},g_j;\d H\_{j-1},\w\_{k-j})
  \end{equation}
  where $\Gamma$ is the cycle given by the flat limit of $\{\~f\_{n-k}=0\}$
  as $e\to0$,
  \begin{equation}
    g_j = \frac{ \bigwedge \d \~f\_{n-k} \bigwedge \w\_{k-j+1} \bigwedge \d H\_{j-1} \wedge(\d e+ce\d H_j) }
             {\d x_1\wedge\cdots\wedge\d x_n\wedge\d e}
  \end{equation}
  and $\~f\_{n-k}$ are as given by Lemma~\ref{lem:smooth-deform}.
\end{Lem}

\begin{proof}
  Applying Lemma~\ref{lem:pfaff-rolle} with $\w\_k$
  and $\d H=\d H_1$ we obtain
  \begin{multline}
    \mult_p(f\_{n-k};\w\_k) \le \mult_p(\~f\_{n-k},g_1;\w\_{k-1})\\
    +\mult_p(\~f\_{n-k};\d H_1,\w\_{k-1})
  \end{multline}
  where the first summand corresponds to the $j=1$ summand
  in~\eqref{eq:mc-inductive-step}. Applying now
  Lemma~\ref{lem:pfaff-rolle} with $\d H_1,\w\_{k-1}$ and $\d H_2$
  we obtain
  \begin{multline}
    \mult_p(\~f\_{n-k};\d H_1,\w\_{k-1}) \le \mult_p(\~f\_{n-k},g_2;\d H_1,\w\_{k-2})\\
    +\mult_p(\~f\_{n-k};\d H_1,\d H_2,\w\_{k-2}).
  \end{multline}
  where the first summand corresponds to the $j=2$ summand
  in~\eqref{eq:mc-inductive-step}. Note that formally one should apply
  a further smoothing deformation to $\~f\_{n-k}$ when applying
  Lemma~\ref{lem:pfaff-rolle}. This would not make any difference for
  the rest of the argument, but to simplify the notation we assume
  from the start that the deformation $\~f\_{n-k}$ was chosen
  sufficiently generic to apply Lemma~\ref{lem:pfaff-rolle} with
  $\d H\_j,\w_{k-j}$ for $j=1,\ldots,k-1$.
  
  Continuing in this manner one obtains the summands corresponding to
  $j=1,\ldots,k-1$ in~\eqref{eq:mc-inductive-step}. In the $j=k-1$
  step, the second summand is given by $\mult_p(\~f\_{n-k},\d H\_k)$.
  Since the Pfaffian system $\d H\_k$ defines a generic affine-linear
  space passing through $p$ and properly intersecting the $k$-cycle
  $\Gamma$ there, this term is equal to $\mult_p\Gamma$.
\end{proof}

We are now ready to present the construction of the multiplicity
cycles. The construction depends on the choice of various parameters,
the totality of which we denote by $\vH$.

\begin{Def}\label{def:mult-cycle}
  The \emph{multiplicity cycle} $\Gamma^\vH(f\_{n-k};\w\_k)$ is a
  mixed cycle in $M$ defined recursively as follows. If $k=0$,
  $\Gamma^\vH(f\_{n-k};\w\_k)$ is the cycle given by the flat limit of
  the set $\{\~f\_n=0\}$ as $e\to0$. Otherwise,
  \begin{equation}
    \Gamma^\vH(f\_{n-k};\w\_k) := \Gamma+\sum_{j=1}^{k-1} \Gamma^\vH(\~f\_{n-k},g_j;\d H\_{j-1},\w\_{k-j})
  \end{equation}
  where $\Gamma$, $g\_{k-1}$ and $\tilde f\_{n-k}$ are as given in
  Lemma~\ref{lem:mc-inductive-step}. In each recursive step we use
  different generic one-forms $\d H_j$ and parameters defining the
  smoothing deformations of $f\_{n-k}$ in
  Lemma~\ref{lem:smooth-deform}, all of which are encoded by $\vH$.
\end{Def}

\begin{proof}[Proof of Theorem~\ref{thm:mult-cycle}]
  Let $\mu:M\to\Z_{\ge0}$ be defined by
  \begin{equation}
    \mu(p) := \mult_p(f\_{n-k};\w\_k)
  \end{equation}
  We denote $\Gamma^\vH:=\Gamma^\vH(f\_{n-k};\w\_k)$. If
  $p_1,\ldots,p_s\in U$ is any finite set of points, then for a
  sufficiently generic choice of $\vH$ (and the parameters $N$ for the
  smoothing deformations sufficiently large) we have
  \begin{equation} \label{eq:mu-vs-mc}
    \mu(p_i)\le\mult_{p_i}\Gamma^\vH \qquad i=1,\ldots,k.
  \end{equation}
  Indeed, for such $\vH$ Lemma~\ref{lem:mc-inductive-step} applies
  with $p=p_i$, and~\eqref{eq:mu-vs-mc} follows by reverse induction
  on $k$. Moreover, one can certainly choose $\vH$ generic enough (and
  with large enough $N$) so that this applies to each of the finitely
  many points under consideration.

  To obtain an estimate uniform in $p$ we appeal to the results
  of~\secref{sec:sc-bounds}. From the construction of $\Gamma^\vH$ it
  is clear that it is an algebraic cycle whose total degree is
  uniformly bounded in terms of the degrees of $f\_{n-k}$ and $\w\_k$
  (for a more precise statement see
  Theorem~\ref{thm:mult-cycle-degs}). Thus, by
  Proposition~\ref{prop:mc-complexity} the function
  $p\to\mult_p\Gamma^\vH$ is an upper semicontinuous function of
  complexity bounded by some uniform constant $D$ independent of
  $\vH$. Thus, by Proposition~\ref{prop:sc-compact} there exists some
  finite set of points $P\subset M$ such that for any $\vH$,
  \begin{equation}
    \mu(p)\rest P \le \mult_p\Gamma^\vH\rest P \implies \mu(p) \le \mult_p\Gamma^\vH \text{ for any $p\in M$}.
  \end{equation}
  Choosing now $\vH$ sufficiently generic so that~\eqref{eq:mu-vs-mc}
  holds for every point of $P$ and setting
  $\Gamma(f\_{n-k},\w\_k):=\Gamma^\vH$ concludes the proof.
\end{proof}

\begin{proof}[Proof of Theorem~\ref{thm:mult-cycle-degs}]
  Suppose that $(f\_{n-k};\w\_k)$ have degrees
  $(\beta\_{n-k};\alpha\_k)$. If $j=n-k$ then $\Gamma^\vH(f\_{n-k},\w\_k)$ is
  equal to the flat limit of a family of cycles defined by
  equations of degrees $\beta\_{n-k}$ and the claim follows.

  Otherwise, Definition~\ref{def:mult-cycle} implies that the
  $j$-dimensional piece of $\Gamma^\vH(f\_{n-k},\w\_k)$ is a sum of
  the corresponding $j$-dimensional pieces of $k$ multiplicity cycles,
  each having degrees bounded by $(\beta\_{n-k},S;\alpha\_{k-1})$.
  Continuing inductively, each such cycle gives rise to $k-1$
  multiplicity cycles, each having degrees bounded by
  $(\beta\_{n-k},S,2S;\alpha\_{k-2})$ and so on. The result follows by
  induction.
\end{proof}

Theorem~\ref{thm:mult-cycle} deals particularly with complete
intersections, i.e. with the situation where the number of equations
$f\_{n-k}$ is complementary to the number of one-forms $\w\_k$. However,
this can easily be extended to more general intersections.

\begin{Cor} \label{cor:mc-general} Let $V\subset M\times\C_e$ be a
  family given by the common zero locus of families of polynomials
  $f\_m$ (not necessarily a complete intersection). Then there exist a
  mixed algebraic cycle $\Gamma:=\Gamma(f\_m;\w\_k)$ of dimension at
  most $k$ in $M$ such that for every $p\in U$,
  \begin{equation}
    \mult_p(\w\_k;V) \le \mult_p\Gamma(f\_m;\w\_k).
  \end{equation}
  Moreover, the cycle is obtained algebraically from $f\_m;\w\_k$ and
  satisfies the same degree bounds as in
  Theorem~\ref{thm:mult-cycle-degs}.
\end{Cor}
\begin{proof}
  Let $\mu:M\to\Z_{\ge0}$ be defined by
  \begin{equation}
    \mu(p) := \mult_p(\w\_k;V).
  \end{equation}
  Let $p\in M$ be any point, and denote by $V_p$ the germ of $V$ at
  $p$. Then $V_p\cap S^k_p$ is a union of curves
  $\gamma^p_1,\ldots,\gamma^p_{r_p}$ projecting dominantly on $\C_e$
  (and possibly some other components), and only these curves
  contribute to the multiplicity defining $\mu(p)$. If we let
  $g^\vL\_{n-k}$ denote $n-k$ generic linear combinations of $f\_m$
  (with $\vL$ denoting the collection of coefficients defining these
  combinations), then by standard arguments
  $\gamma^p_1,\ldots,\gamma^p_{r_p}$ are also irreducible components
  of the intersection $\{g^\vL\_{n-k}=0\}\cap S^k_p$, and therefore
  \begin{equation}
    \mu(p) \le \mu^\vL(p) := \mult_p(\{g^\vL\_{n-k}=0\};\w\_k).
  \end{equation}

  We pick generic $\vH$ and denote
  \begin{equation}
    \Gamma^{\vH,\vL}:=\Gamma^\vH(g^\vL\_{n-k};\w\_k).
  \end{equation}
  Then by Theorem~\ref{thm:mult-cycle} we have
  \begin{equation}
    \mu(p) \le \mu^\vL(p) \le \mult_p \Gamma^{\vH,\vL}.
  \end{equation}
  Moreover, for sufficiently generic choices of $\vH,\vL$ the same
  holds for any finite collection of points $P\subset M$. The proof
  can now be concluded in the same way as in the proof of
  Theorem~\ref{thm:mult-cycle}.
\end{proof}

\section{Applications}
\label{sec:applications}

\subsection{Critical points and Euler characteristics of Milnor fibers}
\label{sec:mf-estimates}

Let $\w\_k\in\Lambda^1(M)$. We suppose for simplicity that they define
an integrable Pfaffian system at every point $p\in M$ with the
corresponding chain of integral submanifolds
$S^k_p\subset\cdots\subset S^1_p\subset S^0_p=M$ (although one could easily relax this
requirement).

For a family of polynomials $f\_m$ on $M$ we define
\begin{equation} \label{eq:sigma-def}
  \S(f\_m;\w\_k) := \{ f\_m=\bigwedge \d f\_m\bigwedge\w\_k\wedge\d e = 0\}.
\end{equation}
In other words, for each fiber $e=\e$ we define $\S(f\_m;\w\_k)$ to be
the set of points $p$ where $\{f^\e\_m=0\}\cap S^k_p$ is not
effectively smooth. Note that~\eqref{eq:sigma-def} is not in general a
complete intersection: it involves the vanishing of an $m+k+1$ form in
$n+1$ variables, and is thus a determinantal variety. However,
its deformation multiplicity at any point $p\in M$ may still be
estimated by Corollary~\ref{cor:mc-general}. We will apply this
result to estimate the topological complexity of Milnor fibers
of deformations.

Consider the flat family defined by
\begin{equation}
  X \subset M\times\C_e ,\qquad X := \cF\big[ \{f\_m=0\} \big].
\end{equation}
Recall that the Milnor fiber of $X$ at a point $p$ is defined
to be
\begin{equation}
  F_p := X^\e\cap D_\delta(p), \qquad 0\ll\e\ll\delta\ll1
\end{equation}
where $D_\delta(p)$ denotes the disc of radius $\delta$ around $p$ in
any analytic coordinate system. The homotopy type of the set $F_p$ is
independent of the choice of coordinates and $\e,\delta$.

We call a point $p\in M$ \emph{good} if the Milnor fiber at $p$ is
effectively smooth. In this case a local complex analog of real Morse
theory allows one to study the topological structure of the Milnor
fiber in terms of critical points of linear functionals\footnote{in
  fact, a more general stratified version holds also without the
  assumption on $p$, see \cite{le:polar-curves}}. This result is known
as Le's attaching formula (see for instance \cite{le:polar-curves};
cf. \cite[Proposition~2]{gabrielov:mult} and
\cite[Proposition~12]{mult-morse}). In particular, this implies the
following.

\begin{Prop} \label{prop:mf-bounds}
  Let $\vell:=\ell\_{n-m+1}$ be generic linear functionals on
  $M=\C^n$, and $p\in M$ a good point. Then for $k=0,\ldots,n-m$,
  \begin{equation}
    b_{n-m-k}(F_p) \le \mult_p (\d\ell\_k; \S(f\_m;\d\ell\_{k+1}) )
  \end{equation}
\end{Prop}
\begin{proof}[Sketch of proof.]
  Denote $X_0:=X$. According to Le's attaching formula, the Milnor
  fiber of $X$ is obtained from the Milnor fiber of
  $X_1:=X\cap\{\ell_1=\ell_1(p)\}$ by attaching $c_{n-m}$ cells, where
  $c_{n-m}$ is given by the number of critical points of $\ell_1$ on
  $X_0^\e$ converging to $p$ as $\e\to0$. Thus
  \begin{equation}
    c_{n-m}=\mult_p \S(f\_m;\d\ell_1).
  \end{equation}
  Similarly, the Milnor fiber of $X_1$ is obtained from the Milnor
  fiber of $X_2:=X_1\cap\{\ell_2=\ell_2(p)\}$ by attaching $c_{n-m-1}$
  cells, where $c_{n-m-1}$ is given by the number of critical points
  of $\ell_2$ on $X_1^\e$ converging to $p$ as $\e\to0$. Thus
  \begin{equation}
    c_{n-m-1} = \mult_p(\d\ell_1; \S(f\_m;\d\ell\_2)).
  \end{equation}
  Continuing inductively and noting that $b_r\le c_r$ concludes the
  proof.
\end{proof}

In combination with Corollary~\ref{cor:mc-general} we have the
following multiplicity cycle estimate for the Betti numbers of
the Milnor fiber of a deformation over a general point.

\begin{Thm} \label{thm:mf-mc} For $r=0,\ldots,n-m$ there exists a
  mixed algebraic cycle $\Gamma_r:=\Gamma_r(X)$ such that for any good
  point $p\in M$,
  \begin{equation}
    b_r(F_p) \le \mult_p \Gamma_r(X).
  \end{equation}
  Moreover, if $\deg f\_m\le (d,\ldots,d)$ then
  \begin{equation}
    \deg \Gamma^{n-j}_r(X) \le D_{n,n-m-r,j} d^{n-j}, \qquad D_{n,k,j}:=m^{n-j} C_{n,k,j}
  \end{equation}
  for $C_{n,k,j}$ given in Corollary~\ref{cor:mc-degs-simple}.
\end{Thm}
\begin{proof}
  For any finite collection of points $P\subset M$ the claim is proved
  by choosing sufficiently generic $\vell$ and applying
  Proposition~\ref{prop:mf-bounds} and Corollary~\ref{cor:mc-general},
  noting that the equations defining $\S(f\_m;\d\ell\_{k+1})$ have
  degrees bounded by $md$ (for any $k$). To obtain a bound uniform
  over all good point $p\in M$ one can argue in the same way as in the
  proof of Theorem~\ref{thm:mult-cycle}.
\end{proof}

\subsection{Multiplicity estimates for non-singular vector fields}
\label{sec:vf-estimates}

Let $\xi$ be a polynomial vector field,
\begin{equation}
  \xi = \xi_1 \pd{}{x_1}+\cdots+\xi_n\pd{}{x_n}, \qquad \xi\_n\in\C[x\_n]
\end{equation}
and $P$ be a polynomial. Let $p\in M$ be a non-singular point of
$\xi$. The \emph{multiplicity of $P$ along $\xi$ at $p$} is defined to
be
\begin{equation}
  \mult_p^\xi(P) := \ord_p P\rest{\gamma_p}
\end{equation}
where $\gamma_p$ denotes the trajectory of $\xi$ through $p$. If $\xi$
is singular at $p$, or if $P$ vanishes identically on $\gamma_p$ we
define $\mult_p^\xi(P)=\infty$.

The problem of estimating $\mult_p^\xi(P)$ in terms of $d:=\deg P$ and
$\delta:=\deg\xi$ has been considered by various authors motivated by
applications in transcendental number theory
\cite{bm:mult-I,bm:mult-II,nesterenko:mult-nonlinear,nesterenko:modular},
control theory
\cite{risler:nonholonomy,gabrielov:mult-old,gabrielov:mult} and
dynamical systems \cite{ny:chains,yomdin:oscillation}. We present
below the strongest estimate known for $\mult_p^\xi(P)$ as a function
of $p$, for arbitrary non-singular vector fields. This result (in a
slightly different formulation) was first obtained using algebraic
methods in \cite{nesterenko:mult-nonlinear} and later using
topological methods in \cite{mult-morse}. We give the formulation of
\cite{mult-morse} as it is more suitable for comparison with the
present paper.

For a (possibly mixed) cycle $\Gamma$ we define a function
$\degf_p(\Gamma):M\to\Z_{\ge0}$ as follows. For an irreducible variety
$V\subset M$ we define
\begin{equation}
  \degf_p(V) :=
  \begin{cases}
    \deg V & p\in V \\
    0 & \text{otherwise},
  \end{cases}
\end{equation}
and extend this to arbitrary cycles $\Gamma$ by linearity. Then the
following holds \cite{mult-morse}.

\begin{Thm}
  There exists a mixed algebraic cycle $\Gamma\subset M$ such that
  \begin{equation}
    \mult_p^\xi(P) \le \degf_p(\Gamma)
  \end{equation}
  whenever the left-hand side is finite, and moreover,
  \begin{equation}
    \deg \Gamma^j \le \tilde C_{n,\delta} d^j
  \end{equation}
  where $\tilde C_{n,\delta}$ is a constant depending
  singly-exponentially on $n$ and $\delta$.
\end{Thm}

The function $\degf_p(\Gamma)$ is somewhat artificial, involving
global properties of the cycle $\Gamma$ when evaluated at a particular
point $p$. The multiplicity function $\mult_p\Gamma$ is a more natural
local analog, and moreover it is clear that
$\mult_p\Gamma\le\degf_p\Gamma$. We will prove the following result.

\begin{Thm}\label{thm:vf-mc}
  There exists a mixed algebraic cycle $\Gamma(P;\xi)\subset M$ such that
  \begin{equation}
    \mult_p^\xi(P) \le \mult_p \Gamma(P;\xi)
  \end{equation}
  whenever the left hand side is finite, and moreover,
  \begin{equation}
    \deg \Gamma^j(P;\xi) \le C_{n,\delta} d^j,
  \end{equation}
  where $C_{n,\delta}$ is a constant depending singly-exponentially on
  $n$ and $\delta$.
\end{Thm}

\begin{Rem} \label{rem:mult-sing}
  During the preparation of this manuscript, the author has found an
  additional, algebraic proof for Theorem~\ref{thm:vf-mc}, cf.
  \cite{mult-sing}. The two results are not comparable: the approach
  of \cite{mult-sing} gives estimates valid also for singular vector
  fields under an additional assumption known as the D-property; on
  the other hand, with this approach the D-property is necessary even
  in the non-singular case in order to produce estimates with explicit
  constants.

  We note that the expression for $\mult_p^\xi(P)$ in terms of local
  multiplicities of cycles was first obtained using the topological
  methods of the present paper. Having this expression in mind, the
  author was later able to produce an inductive algebraic construction
  in \cite{mult-sing}. 
\end{Rem}

The proof of Theorem~\ref{thm:vf-mc} is based on a beautiful
topological characterization of the multiplicity $\mult_p^\xi(P)$
discovered by Gabrielov \cite{gabrielov:mult}, which we now recall.
Let $Q\in\C[x\_n]$ be a polynomial of degree $n-1$, and denote
$P^Q=P+eQ$. For $r=1,\ldots,n$ we define the family
$X_r\subset M\times\C_e$ by
\begin{equation}
  X^r := \cF\big[\{ P^Q = \cdots = \xi^{r-1} P^Q=0  \}\big].
\end{equation}
Denote by $F^r_p$ the Milnor fiber of $X^r$ at the point $p$. Note that
for simplicity of the notation we suppress the dependence of $X^r,F^r_p$
on $Q$.

\begin{Prop}\label{prop:gab-mf}
  Let $p\in M$ be such that $\mult_p^\xi(P)<\infty$. Then
  for a sufficiently generic choice of $Q$,
  \begin{equation}
    \mult_p^\xi(P) = \sum_{r=1}^n \chi(F^r_p).
  \end{equation}
  Moreover, $p$ is a good point of each of the families $X^r$ (i.e.,
  the Milnor fiber $F^r_p$ is effectively smooth).
\end{Prop}

The proof of Theorem~\ref{thm:vf-mc} is now a simple exercise using
Theorem~\ref{thm:mf-mc}.

\begin{proof}[Proof of Theorem~\ref{thm:vf-mc}]
  Let $p\in M$ be such that $\mult_p^\xi(P)<\infty$. Choose $Q$
  sufficiently generic so that Proposition~\ref{prop:gab-mf} applies.
  Let $r=1,\ldots,n$. Applying Theorem~\ref{thm:mf-mc} to $X^r$ we
  obtain for every $q=0,\ldots,n-r$ a cycle $\Gamma_{r,q}$ such that
  \begin{equation}
    b_q(F_p^r) \le \mult_p \Gamma_{r,q}.
  \end{equation}
  Noting that the Euler characteristic is bounded by the sum
  of Betti numbers, and setting $\Gamma_r:=\sum_q\Gamma_{r,q}$
  gives
  \begin{equation}
    \chi(F_p^r) \le \mult_p\Gamma_r.
  \end{equation}
  Finally setting $\Gamma(P;\xi):=\sum_r\Gamma_r$ and using
  Proposition~\ref{prop:gab-mf} gives
  \begin{equation}
    \mult_p^\xi(P) \le \mult_p \Gamma(P;\xi).
  \end{equation}

  The same arguments show that one can find a cycle $\Gamma(P;\xi)$
  satisfying the conditions of the theorem for any finite collection
  of points $S\subset M$. One can then conclude that (for sufficiently
  generic $Q$) the bound is in fact uniform over all $p\in M$ in
  the same way as in the proof of Theorem~\ref{thm:mult-cycle}.
\end{proof}

\subsubsection{Multiplicity estimates for non-singular foliations}
\label{sec:mult-foliations}

Let $\xi\_m$ be $m$ commuting polynomial vector fields defining a
foliation $\cF$ on $M$ and $P\_m$ be $m$ polynomial functions. We
define the multiplicity $\mult_p(P\_m;\cF)$ to be the intersection
multiplicity between the leaf of $\cF$ at $p$ and $\{P\_m=0\}$ (or
infinity if the foliation is singular at $p$, or the intersection is
not proper). In other words, we extend our usual definition for
Pfaffian foliations to the case of arbitrary foliations. Note
however that we intentionally introduce this notion for proper
intersections rather than general deformations as we did in the
Pfaffian case.

It is reasonable to ask whether Theorems~\ref{thm:mult-cycle}
and~\ref{thm:vf-mc} can be generalized to a result about general
foliations, as follows.

\begin{Conj}\label{con:foliation-mc}
  There exists a mixed algebraic cycle $\Gamma(P\_m;\cF)\subset M$
  such that
  \begin{equation}
    \mult_p^\cF(P\_m) \le \mult_p \Gamma(P\_m;\cF)
  \end{equation}
  whenever the left hand side is finite, and moreover,
  \begin{equation}
    \deg \Gamma^j(P\_m;\cF) \le C_{n,\cF} d^{n-j}
  \end{equation}
  where $C_{n,\cF}$ is a constant depending only on $n,\cF$.
\end{Conj}

In \cite{gk:noetherian} this problem was studied for the case of a
single point $p$ (i.e., with estimates not depending on $p$). More
specifically, it was proven that the multiplicity in this case is
bounded by $C_{n,\cF} d^{2n}$ where $C_{n,\cF}$ is some constant. It
appears that in combination with the ideas of \cite{mult-morse} the
approach of \cite{gk:noetherian} can be sharpened to give an estimate
$C_{n,\cF} d^n$, and moreover to prove
Corollary~\ref{con:foliation-mc} in full generality (albeit with worse
constants). However, this result falls outside the scope of this paper
and will appear separately.

We note that the formulation of Corollary~\ref{con:foliation-mc}
applies only to proper intersections, rather than to multiplicities of
deformations as in the case of Theorem~\ref{thm:mult-cycle}. It is in
fact natural to conjecture that the same result would hold for
arbitrary deformations. However, this result may require new ideas. A
conjecture of this type was made in \cite{gk:noetherian} for the case
of a single point $p$ (with a less accurate asymptotic), and proved
under a minor technical condition in \cite{noetherian-def}. But it is
unclear whether this approach can be extended to produce multiplicity
cycles and the precise asymptotic of
Conjecture~\ref{con:foliation-mc}.

\subsection{Multiplicity estimates on group varieties}
\label{sec:group-estimates}

Suppose that our ambient space is a commutative algebraic group
variety $G$ with the associated Lie algebra $\fg$ of invariant vector
fields. Let $P$ be a regular function on $G$. Results imposing
restrictions on the geometry of the set of points where $P$ vanishes
to a given order, with respect to one or several vector fields, are
known as \emph{multiplicity estimates on group varieties}. Several
authors have studied such multiplicity estimates, with important
applications in diophantine approximation and transcendental number
theory.

The area was initiated by Masser and W\"ustholz in the two important
papers \cite{mw:groups1,mw:groups2}. In \cite{mw:groups1}, the authors
consider the zeros of $P$ without prescribing the order along a vector
field. In \cite{mw:groups2}, the authors consider zeros of $P$ of a
prescribed order $T$ in the direction of an invariant vector field
$\xi\in\fg$. Both papers rely on an algebraic technique involving
manipulation of ideals, which was later elaborated to account for
zeros of a prescribed order in the direction of several vector fields
by W\"ustholz \cite{W:multgrps} and strengthened by Philippon in
\cite{philippon:groups}. The result of Philippon is essentially
optimal and we make no attempt at improving it. Instead, our goal is
to show that the language of multiplicity can be applied in this
context.

In \cite{moreau:zeros1,moreau:zeros2}, Moreau gave a simple geometric
proof of the main result of \cite{mw:groups1}. However, the treatment
of multiplicities still required the more general algebraic approach
of \cite{mw:groups2}. In this section we present an extension of
Moreau's geometric approach to the case involving multiplicities,
relying on the notion of a multiplicity cycle. We begin by presenting
a result of \cite{mw:groups1} in the formulation of
\cite{moreau:zeros2}.

For simplicity of the presentation we restrict ourselves to the case
$G=(\C^*)^n$, and present estimates in terms of a pure degree $d$
rather than a sequence of mixed degrees (although the results of this
section extend to arbitrary $G$ and more general notions of degree
with minor modifications). To maintain the conventional notation in
our presentation, we use additive notation for the addition rule in
$G$. Following Moreau, for a set $\Sigma\subset G$ we denote
\begin{align*}
  \Sigma^{(p)} &:= \{\gamma_1+\cdots+\gamma_p : \gamma_i\in\Sigma\}\\
  \Sigma^{(0)} &:= \{0\}.
\end{align*}
For a subgroup $H$ of $G$, we denote by $\#(\Sigma+H)/H$ the size of
the coset space $(\Sigma+H)/H$. With these notations Moreau
\cite{moreau:zeros1} proves:
\begin{Thm}\label{thm:moreau}
  Let $\Sigma\subset G$ be a finite subset containing $0$, and
  $P\in\C[x\_n]$ be a polynomial of degree $d$. Suppose that for any
  proper algebraic subgroup $H\subset G$ we have
  \begin{equation}
    \#(\Sigma+H)/H > d^{\codim_G H}.
  \end{equation}
  Then
  \begin{equation}
    \bigcap_{\gamma\in\Sigma^{(n)}} (-\gamma+\{P=0\}) = \emptyset
  \end{equation}
\end{Thm}

If $V\subset G$ is an irreducible variety, we define its
\emph{$d$-weight} to be $\wt_d V:=\deg V / d^{\codim_G V}$. We extend
this by linearity to arbitrary varieties and cycles in $G$. For any
(effective) cycle $\Gamma$, we denote by $\o\Gamma$ the variety
underlying $\Gamma$ (i.e. without associated multiplicities). With
this notation, Theorem~\ref{thm:moreau} follows from the following
more general assertion.

\begin{Lem}\label{lem:moreau}
  Let $\Sigma\subset G$ be a finite subset containing $0$. Let
  $V \subset G$ be a variety of top dimension $m$ which is cut out
  set-theoretically by polynomials of degree at most $d$. Suppose that
  for every proper algebraic subgroup $H\subset G$ we have
  \begin{equation}\label{eq:coset-bd}
    \#(\Sigma+H)/H > d^{\codim_G H} \wt_d V.
  \end{equation}
  Then
  \begin{equation}\label{eq:coset-cap}
    \bigcap_{\gamma\in\Sigma^{(m+1)}} (-\gamma+V) = \emptyset
  \end{equation}
\end{Lem}
\begin{proof}
  We proceed by induction on $m$. For $m=0$, it follows
  from~\eqref{eq:coset-bd} (with $H=\{0\}$) that
  $\#\Sigma>\deg V=\#V$. Thus for any $g\in G$ there exists
  $\gamma\in\Sigma$ with $g+\gamma\not\in V$,
  proving~\eqref{eq:coset-cap}.

  Assume now that $m>0$ and write $V^m=V_1\cup\cdots\cup V_s$.
  Let $i=1,\ldots,s$ and let $H_i$ denote the stabilizer of $V_i$ in
  $G$. Certainly $\codim_G H_i\ge m$, so by~\eqref{eq:coset-bd} we
  have
  \begin{equation}
    \#(\Sigma+H_i)/H_i > d^m \wt_d V \ge \deg V^m \ge s
  \end{equation}
  It follows that there exists an element $\gamma_i\in\Sigma$ such
  that $\gamma_i+V_i\neq V_j$ for $j=1,\ldots,s$, and hence
  $\gamma_i+V_i\not\subset V$. Since $V$ is cut out by polynomials of
  degree bounded by $d$, we can find such a polynomial $P_i$ which
  vanishes on $V$ but not on $\gamma_i+V_i$. Set
  \begin{equation}
    V'_i=V_i\cap(-\gamma_i+\{P_i=0\}).
  \end{equation}
  By the Bezout theorem, $\wt_dV'_i\le\wt_d V_i$.

  We now set
  \begin{equation}
    V':=\bigcup_{i=1}^s V'_i \cup V^{m-1}\cup\cdots\cup V^0.
  \end{equation}
  Then the top-dimension of $V'$ is at most $m-1$ and
  $\wt_dV'\le\wt_dV$. By construction we have
  \begin{equation}
    \bigcap_{\gamma\in\Sigma} (-\gamma+V) \subset V'.
  \end{equation}
  Finally, by induction we have
  \begin{equation}
    \bigcap_{\gamma\in\Sigma^{(m+1)}} (-\gamma+V) \subset
    \bigcap_{\gamma\in\Sigma^{(m)}} (-\gamma+V') = \emptyset
  \end{equation}
  as claimed.
\end{proof}

Let $V\subset G$ be an irreducible variety and $H$ a divisor. If
$V\not\subset H$ we denote by $[V]*H$ the intersection of $[V]$ and
$H$ as cycles. Otherwise, we denote $[V]*H:=V$. We extend this by
linearity to define $\Gamma*H$ for an arbitrary cycle $\Gamma$. Note
that $*$ is not associative, though it is associative at the level of
the underlying sets. If $\deg H\le d$ then by the Bezout theorem,
$\wt_d \Gamma*H\le\wt_d\Gamma$.

\begin{Lem} \label{lem:mult-nodrop}
  Let $\Gamma$ be a mixed cycle in $G$ and $H$ a divisor. For
  any $p\in H$ we have
  \begin{equation}
    \mult_p \Gamma*H \ge \mult_p\Gamma.
  \end{equation}
\end{Lem}
\begin{proof}
  It suffices to prove the claim for the case $\Gamma=[V]$ for an
  irreducible variety $V\subset G$. If $V\subset H$ then the claim is
  obvious. Otherwise, the right hand side is given by the intersection
  at $p$ of $\dim V$ generic divisors passing through $p$ and the
  right hand side is given by the intersection at $p$ of $H$ and
  $\dim V-1$ generic divisors passing through $p$. The result follows
  by the semicontinuity of the intersection multiplicity.
\end{proof}

We now consider multiplicities in the direction of an invariant vector
field. Let $\xi\in\fg$ be an invariant vector field on $G$ and $P$ a
polynomial of degree $d$. In particular, $\xi$ is a linear vector
field, and derivation with respect to $\xi$ does not increase the
degree of a polynomial\footnote{The same property holds, with a minor
  technical modification, for arbitrary invariant vector fields on
  commutative group varieties}. The trajectories of $\xi$ can be
written as the solution of a Pfaffian chain. Indeed, $\xi$ spans the
kernel of $n-1$ invariant one-forms $\w\_{n-1}\in\fg^*$, and the
integrability condition for these forms follows from the commutativity
of $\fg$. By Theorem~\ref{thm:mult-cycle} there exists a cycle
$\Gamma(P;\xi)$ in $G$ with
\begin{equation}
  \wt_d \Gamma(P;\xi) \le C_G
\end{equation}
for some constant $C_G$ depending only on $G$, and such that for any
$p\in G$ with $\mult_p^\xi P<\infty$ we have
\begin{equation}
  \mult_p^\xi P = \mult_p(P;\w\_{n-1}) \le \mult_p \Gamma(P;\xi).
\end{equation}
Alternatively, one can apply Theorem~\ref{thm:vf-mc} to $\xi$ directly
(giving a somewhat worse constant).

\begin{Prop}\label{prop:vt-weight}
  Let $P$ be a polynomial of degree $d$ and $T\in\N$. Denote by $V_T$
  the variety of points where $P$ vanishes to order at least $T$ along
  $\xi$,
  \begin{equation}
    V_T := \{ P=\xi P = \cdots = \xi^{T-1} P = 0 \}.
  \end{equation}
  Finally let $\hat V_T$ denote the Zariski closure of
  $V_T\setminus V_\infty$. Then $\wt_d \hat V_T\le C_G/T$.
\end{Prop}
\begin{proof}
  Let $\Gamma$ be the cycle given by
  \begin{equation}
    \Gamma := (\cdots(\Gamma(P;\xi)*\{P=0\})\cdots)*\{\xi^{T-1}P=0\}.
  \end{equation}
  By Lemma~\ref{lem:mult-nodrop} we have
  \begin{equation}\label{eq:gamma-wt-bound}
    \wt_d\Gamma\le C_G.
  \end{equation}
  The cycle $\Gamma$ is clearly supported on $V_T$. Let $W$ be an
  irreducible component of $\hat V_T$ and suppose that it appears in
  $\Gamma$ with multiplicity $m_W$. Then for a generic $p\in W$ we
  have $\mult_p^\xi<\infty$, and by Lemma~\ref{lem:mult-nodrop} and
  the remark preceding it we have
  \begin{equation} \label{eq:mw-bound}
    T \le \mult_p^\xi P \le \mult_p \Gamma(P;\xi) \le \mult_p\Gamma = m_W
  \end{equation}
  Finally from~\eqref{eq:gamma-wt-bound} and~\eqref{eq:mw-bound} we
  have $\wt_d \hat V_T\le C_G/T$ as claimed.
\end{proof}

An application of Lemma~\ref{lem:moreau} now gives the following
multiplicity version of Moreau's result, essentially agreeing with the
multiplicity estimate of \cite{mw:groups2} for the case of a single
group $G$.

\begin{Thm}\label{thm:moreau-extended}
  Let $\Sigma\subset G$ be a finite subset containing $0$, and
  $P\in\C[x\_n]$ be a polynomial of degree $d$. Let $T\in\N$ and
  suppose that for every proper algebraic subgroup $H\subset G$ we
  have
  \begin{equation}
    \#(\Sigma+H)/H > d^{\codim_G H}\cdot C_G/T.
  \end{equation}
  If $P$ vanishes at every point of $\Sigma^{(n)}$ with multiplicity
  at least $T$ in the direction of $\xi$, then there exists
  $\gamma\in\Sigma^{(n)}$ such that $P$ vanishes identically at
  $\gamma$ in the direction of $\xi$.
\end{Thm}
\begin{proof}
  Define $V_T,\hat V_T$ as in Proposition~\ref{prop:vt-weight}. By assumption,
  \begin{equation}
    0 \in \bigcap_{\gamma\in\Sigma^{(n)}} (-\gamma+V_T).
  \end{equation}
  On the other hand, by Proposition~\ref{prop:vt-weight} we have
  $\wt_d \hat V_T\le C_G/T$ and Lemma~\ref{lem:moreau} then gives
  \begin{equation}
    \bigcap_{\gamma\in\Sigma^{(n)}} (-\gamma+\hat V_T) = \emptyset.
  \end{equation}
  It follows that for some $\gamma\in\Sigma^{(n)}$ we have
  $\gamma\in V_T\setminus\hat V_T$, which implies the claim of the
  theorem.
\end{proof}

\section{A compactness property for semicontinuous bounds}
\label{sec:sc-bounds}
  
In this section we assume for simplicity of the formulation that the
ambient variety $M$ is given by $M=\C^n$ (although similar results
would hold, nearly verbatim, in a much more general context).

Recall that a function $F:M\to\N$ is said to be (algebraic)
\emph{upper semicontinuous} if the sets
$F_{\ge n}:=F^{-1}([n,\infty))$ are closed varieties for each
$n\in\N$. We will say that $F$ has \emph{complexity bounded by $D$} if
moreover, all of these sets can be defined by equations of degree at
most $D$.

The following proposition appeared in \cite{mult-morse}. We include
the proof for the convenience of the reader.

\begin{Prop} \label{prop:sc-compact}
  Let $D\in N$ and $f:M\to\N$ an arbitrary bounded function. Then
  there exists a finite set of points $P\subset M$ such that for any
  upper semicontinuous function $F$ of complexity bounded by $D$,
  \begin{equation}
    f\rest P \le F\rest P \implies f \le F.
  \end{equation}
\end{Prop}
\begin{proof}
  Denote by $N$ an upper bound for $f$. Then $f\le F$ if and only if
  $f_{\ge i}\subset F_{\ge i}$ for $i=1,\ldots,N$. Thus it will
  suffice to construct a finite set $P_i\subset f_{\ge i}$ such that for any set $S$ of
  complexity bounded by $D$,
  \begin{equation}
    P_i \subset S \implies f_{\ge i}\subset S
  \end{equation}
  and take $P=\cup_{i=1}^N P_i$.

  Let $L$ denote the linear space of polynomials of degree bounded by
  $D$ on $M$. For any $p\in M$ let $\phi_p:L\to\C$ denote the
  functional of evaluation at $p$. Finally, for any set $P\subset M$
  denote by $L_P\subset L$ the linear subspace of polynomials which vanish at
  every point of $P$.

  We need to construct a finite set $P_i\subset f_{\ge i}$ with
  $L_{P_i}=L_{f_{\ge i}}$. This is clearly possible. Indeed,
  $L_{f_{\ge i}}$ is the kernel of the set of functionals
  $\{\phi_p:p\in f_{\ge i}\}$. Since $L_{f_{\ge i}}$ has finite
  codimension in $L$, one can choose a finite subset $P_i$ (in fact,
  of size equal to this codimension) of functionals whose kernel,
  $L_{P_i}$ agrees with $L_{f_{\ge i}}$. This concludes the proof.
\end{proof}

The following simple exercise is left for the reader.

\begin{Lem}\label{lem:sc-sum}
  Let $F_i,i=1,\ldots,N$ be an upper semicontinuous functions with complexity $D_i$
  and bounded by $B_i$. Then $\sum_{i=1}^N F_i$ is an upper semicontinuous
  function with complexity bounded by $D'=D'(D\_N,B\_N)$.
\end{Lem}

Finally, the following proposition allows the application of
Proposition~\ref{prop:sc-compact} to multiplicity cycles.

\begin{Prop} \label{prop:mc-complexity} Let $\Gamma$ be an algebraic
  cycle (possibly of mixed dimension) of total degree bounded by $d$.
  Then $p\to\mult_p\Gamma$ is an upper semicontinuous function of
  complexity bounded by a constant $D$ depending only on $d$.
\end{Prop}
\begin{proof}
  By Lemma~\ref{lem:sc-sum} it is enough to establish the result
  for a cycle of pure dimension $k$. The family of all such cycles
  is parametrized by the projective Chow variety $\cC_{k,d}$. Moreover,
  the correspondence
  \begin{equation}
    \cM_\mu \subset \cC_{k,d}\times M \qquad \cM_\mu=\{ (\Gamma',p):\mult_p\Gamma'\ge\mu \}
  \end{equation}
  is algebraic for $\mu=1,\ldots,d$ and empty for $\mu>d$. In
  particular, the fibers of $\cM_\mu$ under the projection to
  $\cC_{k,d}$ have uniformly bounded degrees (and hence are also
  set-theoretically cut out by equations of uniformly bounded degrees).
  This concludes the proof.
\end{proof}

\let\~=\tildeaccent \let\^=\hataccent
\bibliographystyle{plain}
\bibliography{nrefs}

\def\cprime{$'$}
\begin{thebibliography}{10}

\bibitem{mult-morse}
Gal Binyamini.
\newblock Multiplicity {E}stimates: a {M}orse-theoretic approach.
\newblock {\em Duke Mathematical Journal, to appear}, arXiv:1406.1858, 2014.

\bibitem{mult-sing}
Gal Binyamini.
\newblock Multiplicity estimates, analytic cycles and {N}ewton polytopes.
\newblock {\em {P}reprint}, arXiv:1407.1183, 2014.

\bibitem{noetherian-def}
Gal Binyamini and Dmitry Novikov.
\newblock Multiplicities of {N}oetherian deformations.
\newblock {\em {P}reprint}, arXiv:1406.5959, 2014.

\bibitem{bm:mult-I}
W.~D. Brownawell and D.~W. Masser.
\newblock Multiplicity estimates for analytic functions. {I}.
\newblock {\em J. Reine Angew. Math.}, 314:200--216, 1980.

\bibitem{bm:mult-II}
W.~D. Brownawell and D.~W. Masser.
\newblock Multiplicity estimates for analytic functions. {II}.
\newblock {\em Duke Math. J.}, 47(2):273--295, 1980.

\bibitem{gabrielov:MultPfaffian}
A.~Gabri{\`e}lov.
\newblock Multiplicities of {P}faffian intersections, and the \l ojasiewicz
  inequality.
\newblock {\em Selecta Math. (N.S.)}, 1(1):113--127, 1995.

\bibitem{gabrielov:mult-old}
Andrei Gabrielov.
\newblock Multiplicities of zeroes of polynomials on trajectories of polynomial
  vector fields and bounds on degree of nonholonomy.
\newblock {\em Math. Res. Lett.}, 2(4):437--451, 1995.

\bibitem{gabrielov:mult}
Andrei Gabrielov.
\newblock Multiplicity of a zero of an analytic function on a trajectory of a
  vector field.
\newblock In {\em The {A}rnoldfest ({T}oronto, {ON}, 1997)}, volume~24 of {\em
  Fields Inst. Commun.}, pages 191--200. Amer. Math. Soc., Providence, RI,
  1999.

\bibitem{gk:noetherian}
Andrei Gabrielov and Askold Khovanskii.
\newblock Multiplicity of a {N}oetherian intersection.
\newblock In {\em Geometry of differential equations}, volume 186 of {\em Amer.
  Math. Soc. Transl. Ser. 2}, pages 119--130. Amer. Math. Soc., Providence, RI,
  1998.

\bibitem{GV:PfaffianComplexity}
Andrei Gabrielov and Nicolai Vorobjov.
\newblock Complexity of computations with {P}faffian and {N}oetherian
  functions.
\newblock In {\em Normal forms, bifurcations and finiteness problems in
  differential equations}, volume 137 of {\em NATO Sci. Ser. II Math. Phys.
  Chem.}, pages 211--250. Kluwer Acad. Publ., Dordrecht, 2004.

\bibitem{Khovanskii:Fewnomials}
A.~G. Khovanski{\u\i}.
\newblock {\em Fewnomials}, volume~88 of {\em Translations of Mathematical
  Monographs}.
\newblock American Mathematical Society, Providence, RI, 1991.
\newblock Translated from the Russian by Smilka Zdravkovska.

\bibitem{le:polar-curves}
{L{\^e} D{\~u}ng Tr{\'a}ng}.
\newblock Topological use of polar curves.
\newblock In {\em Algebraic geometry ({P}roc. {S}ympos. {P}ure {M}ath., {V}ol.
  29, {H}umboldt {S}tate {U}niv., {A}rcata, {C}alif., 1974)}, pages 507--512.
  Amer. Math. Soc., Providence, R.I., 1975.

\bibitem{mw:groups1}
D.~W. Masser and G.~W{\"u}stholz.
\newblock Zero estimates on group varieties. {I}.
\newblock {\em Invent. Math.}, 64(3):489--516, 1981.

\bibitem{mw:groups2}
D.~W. Masser and G.~W{\"u}stholz.
\newblock Zero estimates on group varieties. {II}.
\newblock {\em Invent. Math.}, 80(2):233--267, 1985.

\bibitem{moreau:zeros1}
Jean-Charles Moreau.
\newblock D\'emonstrations g\'eom\'etriques de lemmes de z\'eros. {I}.
\newblock In {\em Seminar on number theory, {P}aris 1981--82 ({P}aris,
  1981/1982)}, volume~38 of {\em Progr. Math.}, pages 201--205. Birkh\"auser
  Boston, Boston, MA, 1983.

\bibitem{moreau:zeros2}
Jean-Charles Moreau.
\newblock D\'emonstrations g\'eom\'etriques de lemmes de z\'eros. {II}.
\newblock In {\em Diophantine approximations and transcendental numbers
  ({L}uminy, 1982)}, volume~31 of {\em Progr. Math.}, pages 191--197.
  Birkh\"auser Boston, Boston, MA, 1983.

\bibitem{nesterenko:mult-nonlinear}
Yu.~V. Nesterenko.
\newblock Estimates for the number of zeros of certain functions.
\newblock In {\em New advances in transcendence theory ({D}urham, 1986)}, pages
  263--269. Cambridge Univ. Press, Cambridge, 1988.

\bibitem{nesterenko:modular}
Yu.~V. Nesterenko.
\newblock Modular functions and transcendence questions.
\newblock {\em Mat. Sb.}, 187(9):65--96, 1996.

\bibitem{ny:chains}
D.~Novikov and S.~Yakovenko.
\newblock Trajectories of polynomial vector fields and ascending chains of
  polynomial ideals.
\newblock {\em Ann. Inst. Fourier (Grenoble)}, 49(2):563--609, 1999.

\bibitem{philippon:groups}
Patrice Philippon.
\newblock Lemmes de z\'eros dans les groupes alg\'ebriques commutatifs.
\newblock {\em Bull. Soc. Math. France}, 114(3):355--383, 1986.

\bibitem{risler:nonholonomy}
Jean-Jacques Risler.
\newblock A bound for the degree of nonholonomy in the plane.
\newblock {\em Theoret. Comput. Sci.}, 157(1):129--136, 1996.
\newblock Algorithmic complexity of algebraic and geometric models (Creteil,
  1994).

\bibitem{W:multgrps}
G.~W{\"u}stholz.
\newblock Multiplicity estimates on group varieties.
\newblock {\em Ann. of Math. (2)}, 129(3):471--500, 1989.

\bibitem{yomdin:oscillation}
Y.~Yomdin.
\newblock Oscillation of analytic curves.
\newblock {\em Proc. Amer. Math. Soc.}, 126(2):357--364, 1998.

\end{thebibliography}

\end{document}